\documentclass[12pt]{article}
\usepackage{amsthm,amsmath,amsfonts,amssymb,xcolor,array}
\RequirePackage{graphicx}
\usepackage{caption}
\usepackage{subcaption}
\usepackage{enumerate}
\usepackage{natbib}
\usepackage{url} 

\newcommand{\blind}{1}

\newtheorem{theorem}{Theorem}

\newcolumntype{L}{>{\displaystyle}l}
\newcolumntype{C}{>{\displaystyle}c}

\begin{document}

\def\spacingset#1{\renewcommand{\baselinestretch}%
{#1}\small\normalsize} \spacingset{1}


\if1\blind
{
  \title{\bf On the symmetric and skew-symmetric K-distributions}
  \author{Stylianos E. Trevlakis, Nestor D. Chatzidiamantis, and \\
  	George K. Karagiannidis \\ 
  	\\
    Department of Electrical and Computer Engineering, \\
    Aristotle University of Thessaloniki, Thessaloniki, 54124, Greece}
  \maketitle
} \fi

\if0\blind
{
  \bigskip
  \bigskip
  \bigskip
  \begin{center}
    {\LARGE\bf On the symmetric and skew-symmetric K-distributions}
\end{center}
  \medskip
} \fi

\bigskip
\begin{abstract}
	We propose a family of four-parameter distributions that contain the K-distribution as special case. The family is derived as a mixture distribution that uses the three-parameter reflected Gamma distribution as parental and the two-parameter Gamma distribution as prior. Properties of the proposed family are investigated as well; these include probability density function, cumulative distribution function, moments, and cumulants. The family is termed \textit{symmetric K-distribution} (SKD) based on its resemblance to the K-distribution as well as its symmetric nature. The standard form of the SKD, which often proves to be an adequate model, is also discussed. Moreover, an order statistics analysis is provided as well as the distributions of the product and ratio of two independent and identical SKD random variables are derived. Finally, a generalisation of the proposed family, which enables non-zero skewness values, is investigated, while both the SKD and the skew-SKD are proven capable of describing the complex dynamics of machine learning, Bayesian analysis and other fields through simplified expressions with high accuracy. 
\end{abstract}

\noindent%
\vfill

\newpage
\section{Introduction}
Mixture distributions enable the modeling of complex dynamics through more simplified mathematical expressions. This way they offer tractable solutions in applications where the individual weighted components of the problem exhibit different characteristics. Mixture distributions play an important role in several scientific fields, such as machine learning, Bayesian analysis, wireless communications, and econometric models (e.g.,~\cite{boulogeorgos2020MLinNanoBio,pedersen2015sparse,8941853,withers2012generalized,TrevlakisOvsE,10.1111/ectj.12068}). A \textit{discrete} or \textit{finite} mixture distribution is a linear combination of two or more distributions, i.e.
\begin{equation}
	f(x)=\sum_{i=1}^{n} \omega_{i} f_{i}(x) ; 0<\omega_{i} < 1 ; \sum_{i=1}^{n} \omega_{i}=1
\end{equation}
where the weights $\omega_i$ can be considered as the probabilities of $n$ sets of the distribution parameters, scale, shape and location. \textit{Gaussian discrete mixtures} are the most well-known and can be applied in several fields as for example in signal processing.

If the probabilities $\omega_1$ are continuous random variables (RVs) and belong to another distribution, then the resulting mixture is termed \textit{continuous mixture distribution}. Next, we focus on the case where one of the parameters, denoted by $\Theta$, is random. Then, the conditional probability density function (PDF) $f _X (x | \theta)$ is termed \textit{parental distribution} and the PDF of $\Theta$, denoted by $f_\Theta(\theta)$, is termed \textit{prior distribution}. In this case the joint PDF of $X$ and $\Theta$ is
\begin{equation}
	f(x, \theta)=f_{X}(x | \theta) f_{\Theta}(\theta)
\end{equation}
and the marginal PDF of $X$ is
\begin{equation} \label{mixture_distribution}
	f(x)=\int f(x, \theta) \mathrm{d} \theta=\int f_{X}(x| \theta) f_{\Theta}(\theta) \mathrm{d} \theta.
\end{equation}
A comprehensive summary of discrete and continuous mixture distributions can be found in~\cite{BookKotz}.

Since its introduction by~\cite{jakeman1976model}, the K-distribution has proved to be remarkably useful for modeling the complex dynamics of various systems, such as wireless communications channel modeling and radar applications (e.g.,~\cite{1633320,wang2019secrecy,7406765}). Over the last couple of years, various derivative distributions of the K distribution have attracted attention due to their successful use in machine learning-based ultrasound image reconstruction, as well as heat transfer (e.g.,~\cite{ZHOU2020106001,liu2020effects}).

Both the symmetric counterpart of the K-distribution, namely SKD, and the skew-SKD that are introduced in this work exhibit flexibility in a wide variety of applications, including data fitting, strength-stress modeling, Bayesian learning, and more (e.g.,~\cite{azzalini1999statistical,gupta2001reliability,MUDHOLKAR2000291}). In this paper, we introduce the 4-parameter symmetric K-distribution (SKD) as  the mixture of  the parental 3-parameter reflected Gamma  and the prior 2-parameter Gamma distributions. We derive the PDF, cumulative distribution function (CDF), moments, cumulants, order statistics, as well as the product and ratio distributions. Moreover, we present and study the skew-SKD.

\textit{Notations:} Throughout, let $f\left(\cdot\right)$ and $F\left(\cdot\right)$ stand for the PDF and CDF of the given RV. Write $\left(a\right)_n$ for the Pochhammer Symbol. Denote by $\exp\left(\cdot\right)$ the exponential function, $\Gamma\left(\cdot\right)$ the Gamma function, $G\left(\cdot\right)$ the Meijer-G function, $_p\!F_q\left(\cdot\right)$ for the generalised hypergeometric function, $K_v\left(\cdot\right)$ for the K-function, and $\mathrm{sgn}\left(\cdot\right)$ the sign function.

The rest of this paper is organised as follows. Section~\ref{S:SKD} introduces the four-parameter SKD alongside the formulas for its PDF, CDF, moments, cumulants, order statistics, product and ratio distributions. The skew-SKD is defined in Section~\ref{S:skew-SKD}. Finally, concluding remarks are provided in~\ref{S:Conclusions}.

\begin{figure}
	\centering
	\begin{subfigure}[b]{0.49\textwidth}
		\centering
		\includegraphics[width=\textwidth]{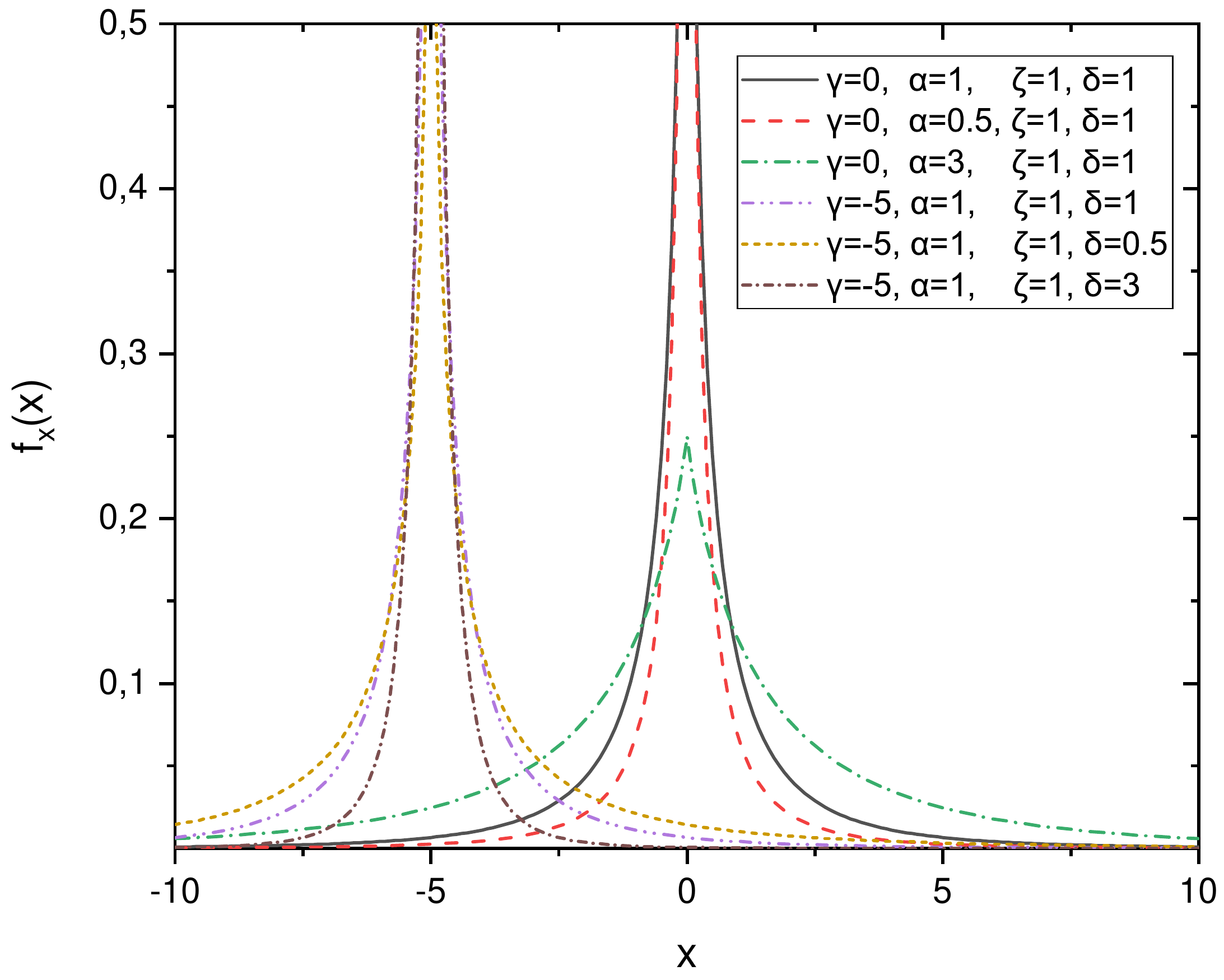}
		\label{fig:PDF}
	\end{subfigure}
	\hfill
	\begin{subfigure}[b]{0.49\textwidth}
		\centering
		\includegraphics[width=\textwidth]{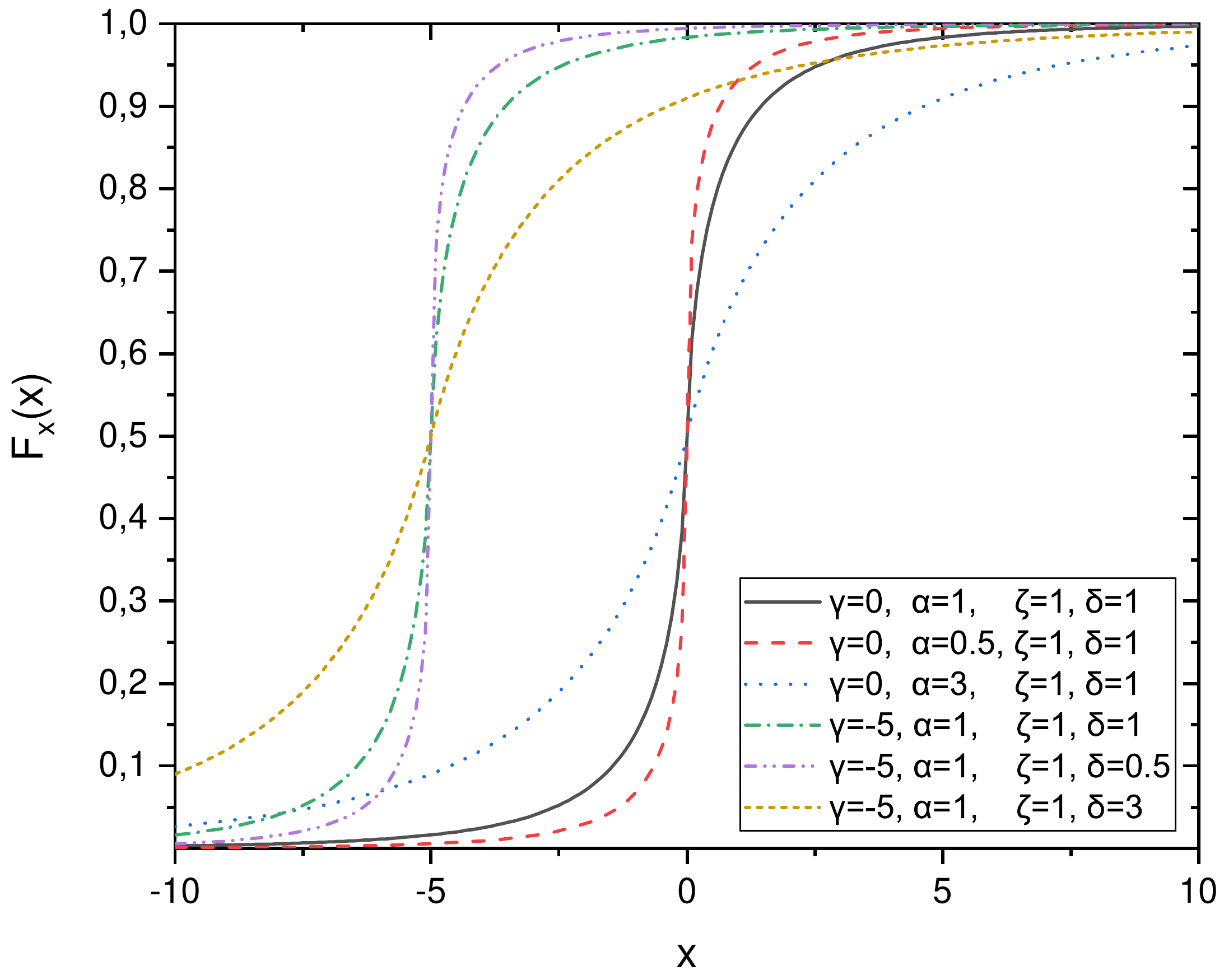}
		\label{fig:CDF}
	\end{subfigure}
	\caption{The (a) PDF and (b) CDF of the SKD for different parameter values.}
\end{figure}

\section{The symmetric K-distribution} \label{S:SKD}
We introduce the SKD as the mixture of  the parental three-parameter reflected Gamma (RG) and the prior two-parameter Gamma distributions. The two-parameters Gamma PDF is given by~\cite[eq. 17.15]{kotz2016continuous}
\begin{equation}\label{eq:Gamma_pdf}
	f(x,\delta,\zeta)=\frac{\delta^{\zeta}x^{\zeta-1}}{\Gamma(\zeta)} \exp (-\delta x), \quad \delta, \zeta >0,
\end{equation}
where $\delta$ and $\zeta$ are the scale and shape parameters, correspondingly. Also, the three-parameters RG PDF is given as in~\cite{BookKotz}
\begin{equation}\label{eq:RG_pdf}
	f(x, \alpha,\beta,\gamma)= \frac{|x-\gamma|^{\alpha-1}}{2\beta^{\alpha} \Gamma(\alpha)} \exp \left(-\frac{|x-\gamma|}{\beta}\right), \quad \alpha, \beta>0,
\end{equation}
where $\alpha, \beta$ and $\gamma$ are the shape, scale and location parameters, correspondingly, while, for $\gamma=0$ and $\beta=1$ the standard form of the RG has the PDF
\begin{align}
	f(x ; \alpha)=\frac{1}{2 \Gamma(\alpha)}|x|^{\alpha-1} e^{-|x|}
\end{align}
and CDF
\begin{equation}
	F(x, \alpha)=\frac{1}{2}-\mathrm{sgn}(x)\frac{\Gamma(\alpha,|x|)}{2 \Gamma(\alpha)}
\end{equation}
where $\mathrm{sgn}(x)$ denotes the sign function and is given by 
\begin{align}
	\mathrm{sgn}(x)=\left\{\begin{array}{ll}
		1, & x \geq 0 \\
		-1, & x<0
	\end{array}\right. .
\end{align}

\subsection{PDF}
According to~\eqref{mixture_distribution}, the PDF of the SKD can be evaluated as follows:
\begin{align}\label{Eq:PDF_1}
	\begin{split}
		f(x)&=\int_{0}^{\infty} \frac{1}{2} \frac{|x-\gamma|^{\alpha-1}}{\beta^{\alpha} \Gamma(\alpha)} \exp \left(-\frac{|x-\gamma|}{\beta}\right) \frac{\delta^{\zeta} \beta^{\zeta-1}}{\Gamma(\zeta)} \exp (-\delta \beta) d\beta \\
		&= \frac{\delta^{\zeta}}{2 \Gamma\left(\alpha\right) \Gamma\left(\zeta\right)} \left|x - \gamma\right|^{\alpha - 1}  \int_{0}^{\infty} \beta^{\zeta - \alpha - 1} \exp \left(-\frac{|x-\gamma|}{\beta} -\delta \beta\right) d\beta .
	\end{split}
\end{align}		
Given the fact that $|x-\gamma|,\delta \in \Re$ and $|x-\gamma|>0$, we can express the integrands according to~\cite[Eq. 2.3.16.1]{B:PrudnI}. Thus, the PDF of the four-parameter SKD can be written as
\begin{equation} \label{Eq:PDF_2}
	f(x)=\frac{\delta^{\frac{\alpha + \zeta}{2}}}{\Gamma\left(\alpha\right) \Gamma\left(\zeta\right)} \left|x - \gamma\right|^{\frac{\alpha + \zeta}{2} - 1} K_{\alpha - \zeta} \left(2\sqrt{\delta \left|x - \gamma\right|}\right) ,
\end{equation}
where $\delta$ and $\gamma$ are the scale and location parameters, and $\alpha$, $\zeta$ are shape parameters. 

For $\gamma=0$ and $\delta=1$, the PDF of the standard form of the SKD is given by
\begin{equation} \label{Eq:PDF_10}
	f_s(x)=\frac{1}{\Gamma\left(\alpha\right) \Gamma\left(\zeta\right)} \left|x\right|^{\frac{\alpha + \zeta}{2} - 1} K_{\alpha - \zeta} \left(2\sqrt{ \left|x\right|}\right) ,
\end{equation}

\subsection{CDF}
The CDF of the SKD can be derived directly from the PDF through 
\begin{align} \label{Eq:CDF_1}
	F\left(x\right) = \int_{-\infty}^{x} f\left(x\right)dx.
\end{align}
By using~\eqref{Eq:PDF_1} it can be rewritten as
\begin{align} \label{Eq:CDF_2}
	F\left(x\right) = \left\{
	\begin{array}{LL}
		\begin{array}{LL}
			\int_{\gamma}^{x} &\frac{\delta^{\frac{\alpha+\zeta}{2}}}{\Gamma(\alpha) \Gamma(\zeta)} \left(x-\gamma\right)^{\frac{\alpha+\zeta}{2}-1} K_{\alpha-\zeta}(2 \sqrt{\delta(x-\gamma)})dx \\
			&+ \int_{-\infty}^{\gamma} \frac{\delta^{\frac{\alpha+\zeta}{2}}}{\Gamma(\alpha) \Gamma(\zeta)} \left(\gamma-x\right)^{\frac{\alpha+\zeta}{2}-1} K_{\alpha-\zeta}(2 \sqrt{\delta(\gamma-x)})dx
		\end{array} &,\quad x \geq \gamma \\
		\int_{-\infty}^{x} \frac{\delta^{\frac{\alpha+\zeta}{2}}}{\Gamma(\alpha) \Gamma(\zeta)} \left(\gamma-x\right)^{\frac{\alpha+\zeta}{2}-1} K_{\alpha-\zeta}(2 \sqrt{\delta(\gamma-x)})dx &,\quad x \leq \gamma
	\end{array} 
	\right. .
\end{align}
From~\eqref{Eq:CDF_2}, we need to calculate three integrals, namely $I_4$, $I_5$, and $I_6$. $I_4$ can be transformed based on $y=x-\gamma>0$ into
\begin{align} \label{Eq:I_4_1}
	I_4 = \frac{\delta^{\frac{\alpha+\zeta}{2}}}{\Gamma(\alpha) \Gamma(\zeta)} \int_{0}^{x-\gamma} y^{\frac{\alpha+\zeta}{2}-1} K_{\alpha-\zeta}(2 \sqrt{\delta y})dy .
\end{align}
Next, the $K_v$ function can be written in terms of the Meijer-G function as in~\cite[eq. 8.4.23.1]{B:PrudnIII} and, thus,~\eqref{Eq:I_4_1} can be equivalently written as
\begin{align} \label{Eq:I_4_2}
	I_4 = \frac{\delta^{\frac{\alpha+\zeta}{2}}}{2 \Gamma(\alpha) \Gamma(\zeta)} \int_{0}^{x-\gamma} y^{\frac{\alpha+\zeta}{2}-1} G_{0,2}^{2,0} \left[\delta y \ \vline \ \begin{array}{cc}-\\ \frac{\alpha-\zeta}{2},-\frac{\alpha-\zeta}{2}\end{array}\right] dy .
\end{align}
After performing the integration in~\eqref{Eq:I_4_2} based on~\cite[eq. 26]{Adamchik1990}, it can be rewritten as
\begin{align}
	I_4 = \frac{\delta^{\frac{\alpha + \zeta}{2}}}{\Gamma(\alpha) \Gamma(\zeta)}(x-\gamma)^{\frac{\alpha + \zeta}{2}} \quad G_{1,3}^{2,1} \left[\delta (x-\gamma) \ \vline \ \begin{array}{cc}1-\frac{\alpha+\zeta}{2}\\ \frac{\alpha-\zeta}{2},-\frac{\alpha-\zeta}{2},-\frac{\alpha+\zeta}{2}\end{array}\right] .
\end{align}

Furthermore, after transforming the integral $I_5$ based on $z=\gamma-x<0$ can be written as
\begin{align} \label{Eq:I_5_1}
	I_5 = \frac{\delta^{\frac{\alpha+\zeta}{2}}}{\Gamma(\alpha) \Gamma(\zeta)} \int_{0}^{\infty} z^{\frac{\alpha+\zeta}{2}-1} K_{\alpha-\zeta}(2 \sqrt{\delta z}) dz ,
\end{align}
which, by expressing the $K_v$ function in terms of Meijer-G as in~\cite[eq. 8.4.23.1]{B:PrudnIII}, can be rewritten as
\begin{align} \label{Eq:I_5_2}
	I_5 = \frac{\delta^{\frac{\alpha+\zeta}{2}}}{2 \Gamma(\alpha) \Gamma(\zeta)} \int_{0}^{\infty} z^{\frac{\alpha+\zeta}{2}-1} G_{0,2}^{2,0} \left[\delta z \ \vline \ \begin{array}{cc}-\\ \frac{\alpha-\zeta}{2},-\frac{\alpha-\zeta}{2}\end{array}\right] dz ,
\end{align}
and, since $\arg \delta < \pi$ and $\delta \neq 0$, we can perform the integration based on~\cite[eq. 2.24.2.1]{B:PrudnIII}. Thus,~\eqref{Eq:I_5_2} can be equivalently written as
\begin{align} \label{Eq:I_5_3}
	\begin{split}
		I_5 = &\frac{\delta^{\frac{\alpha+\zeta}{2}}}{2 \Gamma(\alpha) \Gamma(\zeta)} \delta^{-\frac{\alpha+\zeta}{2}} \Gamma\left[\begin{array}{cc}\frac{\alpha-\zeta}{2}+\frac{\alpha+\zeta}{2},-\frac{\alpha-\zeta}{2}+\frac{\alpha+\zeta}{2}\\ -\end{array}\right] \\
		= & \frac{1}{2 \Gamma(\alpha) \Gamma(\zeta)} \Gamma\left[\frac{\alpha-\zeta}{2}+\frac{\alpha+\zeta}{2}\right] \Gamma\left[-\frac{\alpha-\zeta}{2}+\frac{\alpha+\zeta}{2}\right] = \frac{1}{2} 
	\end{split}
\end{align}

For the third integral, $I_6$, after applying the transformation $k=\gamma-x>0$ and writing the $K_v$ function in terms of Meijer-G as in~\cite[eq. 8.4.23.1]{B:PrudnIII}, it can be expressed as
\begin{align} \label{Eq:I_6_1}
	I_6 = \frac{\delta^{\frac{\alpha+\zeta}{2}}}{2 \Gamma(\alpha) \Gamma(\zeta)} \int_{\gamma-x}^{\infty} k^{\frac{\alpha+\zeta}{2}-1} G_{0,2}^{2,0} \left[\delta k \ \vline \ \begin{array}{cc}-\\ \frac{\alpha-\zeta}{2},-\frac{\alpha-\zeta}{2}\end{array}\right] dk ,
\end{align}
which can be rewritten as
\begin{align} \label{Eq:I_6_2}
	I_6 = I_5 - \frac{\delta^{\frac{\alpha+\zeta}{2}}}{2 \Gamma(\alpha) \Gamma(\zeta)} \int_{0}^{\gamma-x} k^{\frac{\alpha+\zeta}{2}-1} G_{0,2}^{2,0} \left[\delta k \ \vline \ \begin{array}{cc}-\\ \frac{\alpha-\zeta}{2},-\frac{\alpha-\zeta}{2}\end{array}\right] dk ,
\end{align}
and by performing the integration based on~\cite[eq. 26]{Adamchik1990},~\eqref{Eq:I_6_2} can be equivalently expressed as
\begin{align} \label{Eq:I_6_3}
	\frac{1}{2}-\frac{\delta^{\frac{\alpha + \zeta}{2}}}{2\Gamma(\alpha) \Gamma(\zeta)}(\gamma-x)^{\frac{\alpha + \zeta}{2}} \quad G_{1,3}^{2,1} \left[\delta (\gamma-x) \ \vline \ \begin{array}{cc}1-\frac{\alpha+\zeta}{2}\\ \frac{\alpha-\zeta}{2},-\frac{\alpha-\zeta}{2},-\frac{\alpha+\zeta}{2}\end{array}\right] .
\end{align}

Thus, the CDF of the SKD can be written as
\begin{align} \label{Eq:CDF_3}
	F\left(x\right) = \left\{
	\begin{array}{LL}
		\frac{1}{2}+\frac{\delta^{\frac{\alpha + \zeta}{2}}}{\Gamma(\alpha) \Gamma(\zeta)}(x-\gamma)^{\frac{\alpha + \zeta}{2}} \quad G_{1,3}^{2,1} \left[\delta (x-\gamma) \ \vline \ \begin{array}{cc}1-\frac{\alpha+\zeta}{2}\\ \frac{\alpha-\zeta}{2},-\frac{\alpha-\zeta}{2},-\frac{\alpha+\zeta}{2}\end{array}\right] &,\quad x \geq \gamma \\
		\frac{1}{2}-\frac{\delta^{\frac{\alpha + \zeta}{2}}}{2\Gamma(\alpha) \Gamma(\zeta)}(\gamma-x)^{\frac{\alpha + \zeta}{2}} \quad G_{1,3}^{2,1} \left[\delta (\gamma-x) \ \vline \ \begin{array}{cc}1-\frac{\alpha+\zeta}{2}\\ \frac{\alpha-\zeta}{2},-\frac{\alpha-\zeta}{2},-\frac{\alpha+\zeta}{2}\end{array}\right] &,\quad x \leq \gamma
	\end{array} 
	\right. .
\end{align}
and after some simplifications the CDF of the four-parameter SKD can be rewritten as
\begin{align} \label{Eq:CDF_final}
	F\left(x\right) = \frac{1}{2}+\frac{\mathrm{sgn}\left(x-\gamma\right) \delta^{\frac{\alpha + \zeta}{2}}}{2\Gamma(\alpha) \Gamma(\zeta)}|x-\gamma|^{\frac{\alpha + \zeta}{2}} \quad G_{1,3}^{2,1} \left(\delta |x-\gamma| \ \vline \ \begin{array}{cc}1-\frac{\alpha+\zeta}{2}\\ \frac{\alpha-\zeta}{2},-\frac{\alpha-\zeta}{2},-\frac{\alpha+\zeta}{2}\end{array}\right) .
\end{align}

Furthermore, when $\alpha \neq \zeta$, the Meijer-G function in \eqref{Eq:CDF_final} can be expressed in terms of the more familiar $_1\!F_2$ hypergeometric function. Also, since $\arg \delta \left|x-\gamma\right| < \pi$, the CDF can be written as~\cite[eq. 8.2.2.3]{B:PrudnIII}
\begin{align}
	\begin{split}
		F\left(x\right) = &\frac{1}{2}+\frac{\mathrm{sgn}\left(x-\gamma\right) \delta^{\frac{\alpha + \zeta}{2}}}{2\Gamma(\alpha) \Gamma(\zeta)}|x-\gamma|^{\frac{\alpha + \zeta}{2}} \pi \csc (\pi  (\zeta -\alpha )) \\
		&\left(\frac{(\delta  \left| x-\gamma \right| )^{\frac{\alpha-\zeta}{2}}}{\alpha \left(\alpha-\zeta\right) \Gamma(\alpha-\zeta)} \, _1\!F_2(\alpha ;\alpha -\zeta +1,\alpha +1;\delta  \left| x-\gamma \right| ) \right.\\
		&\left. -\frac{(\delta  \left| x-\gamma \right| )^{\frac{\zeta-\alpha}{2}}}{\zeta \left(\zeta-\alpha\right) \Gamma(\zeta-\alpha)}  \, _1\!F_2(\zeta ;-\alpha +\zeta +1,\zeta +1;\delta  \left| x-\gamma \right| )\right) 
	\end{split}
\end{align}
For $\gamma=0$ and $\delta=1$ the CDF of the standard form of the SKD is given by
\begin{align} \label{Eq:CDF_final_10}
	F\left(x\right) = \frac{1}{2}+\frac{\mathrm{sgn}\left(x\right) \delta^{\frac{\alpha + \zeta}{2}}}{2\Gamma(\alpha) \Gamma(\zeta)}|x|^{\frac{\alpha + \zeta}{2}} \quad G_{1,3}^{2,1} \left[ |x| \ \vline \ \begin{array}{cc}1-\frac{\alpha+\zeta}{2}\\ \frac{\alpha-\zeta}{2},-\frac{\alpha-\zeta}{2},-\frac{\alpha+\zeta}{2}\end{array}\right] .
\end{align}

\subsection{N-th moment}
The n-th Moment of the SKD can be derived as
\begin{align} \label{Eq:N_Moment_1}
	\mu_n = \int_{-\infty}^{\infty} x^n f\left(x\right)dx .
\end{align}
By substituting~\eqref{Eq:PDF_2} into~\eqref{Eq:N_Moment_1}, it can be rewritten as
\begin{align} \label{Eq:N_Moment_2}
	\begin{split}
		\mu_n = &\frac{\delta^{\frac{\alpha + \zeta}{2}}}{\Gamma\left(\alpha\right) \Gamma\left(\zeta\right)} \int_{-\infty}^{\gamma} x^n \left(\gamma-x\right)^{\frac{\alpha + \zeta}{2} - 1} K_{\alpha - \zeta} \left(2\sqrt{\delta \left(\gamma-x\right)}\right) dx \\
		&+ \frac{\delta^{\frac{\alpha + \zeta}{2}}}{\Gamma\left(\alpha\right) \Gamma\left(\zeta\right)} \int_{\gamma}^{\infty} x^n \left(x-\gamma\right)^{\frac{\alpha + \zeta}{2} - 1} K_{\alpha - \zeta} \left(2\sqrt{\delta \left(x-\gamma\right)}\right) dx ,
	\end{split} 
\end{align}
which, by using the transformations $y=x-\gamma>0$ and $k=\gamma-x>0$, it can be equivalently written as
\begin{align} \label{Eq:N_Moment_3}
	\begin{split}
		\mu_n = &\frac{\delta^{\frac{\alpha + \zeta}{2}}}{\Gamma\left(\alpha\right) \Gamma\left(\zeta\right)} \int_{0}^{\infty} \left(\gamma-k\right)^n k^{\frac{\alpha + \zeta}{2} - 1} K_{\alpha - \zeta} \left(2\sqrt{\delta k}\right)  dk \\
		&+ \frac{\delta^{\frac{\alpha + \zeta}{2}}}{\Gamma\left(\alpha\right) \Gamma\left(\zeta\right)} \int_{0}^{\infty} \left(\gamma+y\right)^n y^{\frac{\alpha + \zeta}{2} - 1} K_{\alpha - \zeta} \left(2\sqrt{\delta y}\right) dy .
	\end{split} 
\end{align}
Next, by using~\cite[eq. 1.111]{Gradshteyn2014},~\eqref{Eq:N_Moment_3} can be rewritten as
\begin{align} \label{Eq:N_Moment_4}
	\begin{split}
		\mu_n = &\frac{\delta^{\frac{\alpha + \zeta}{2}}}{\Gamma\left(\alpha\right) \Gamma\left(\zeta\right)} \sum_{i=0}^{n} \left(\begin{array}{c}n\\ i\end{array}\right) \gamma^i \left(-1\right)^{n-i} \int_{0}^{\infty} k^{\frac{\alpha + \zeta}{2} - 1 + n - i} K_{\alpha - \zeta} \left(2\sqrt{\delta k}\right)  dk \\
		&+ \frac{\delta^{\frac{\alpha + \zeta}{2}}}{\Gamma\left(\alpha\right) \Gamma\left(\zeta\right)} \sum_{i=0}^{n} \left(\begin{array}{c}n\\ i\end{array}\right) \gamma^i \int_{0}^{\infty} y^{\frac{\alpha + \zeta}{2} - 1 + n - i} K_{\alpha - \zeta} \left(2\sqrt{\delta y}\right) dy .
	\end{split} 		
\end{align}
Furthermore, after using the substitution $\eta = \sqrt{y} = \sqrt{k}$,~\eqref{Eq:N_Moment_4} can be expressed as
\begin{align} \label{Eq:N_Moment_5}
	\mu_n = \frac{2\delta^{\frac{\alpha + \zeta}{2}}}{\Gamma\left(\alpha\right) \Gamma\left(\zeta\right)} \sum_{i=0}^{n} \left(\begin{array}{c}n\\ i\end{array}\right) \gamma^i \left(\left(-1\right)^{n-i} + 1\right) \int_{0}^{\infty} \eta^{\left(\alpha + \zeta + 2n - 2i\right) - 1} K_{\alpha - \zeta} \left(2\sqrt{\delta} \eta\right)  d\eta ,		
\end{align}
and, since $\left(\alpha + \zeta + 2n - 2i\right), \frac{1}{2\sqrt{\delta}}>0$, by employing~\cite[eq. 2.16.2.2]{B:PrudnIII}, \eqref{Eq:N_Moment_5} can be rewritten as
\begin{align} \label{Eq:N_Moment_6}
	\mu_n = \frac{\delta^{i-n}}{2\Gamma\left(\alpha\right) \Gamma\left(\zeta\right)} \sum_{i=0}^{n} \left(\begin{array}{c}n\\ i\end{array}\right) \gamma^i \left(\left(-1\right)^{n-i} + 1\right) \Gamma\left(\alpha + n - i\right) \Gamma\left(\zeta + n - i\right) .		
\end{align}
Next, the previous equation can be simplified as
\begin{align} \label{Eq:N_Moment_7}
	\mu_n =\sum_{k=0}^{n} \frac{\Gamma(n+1) \Gamma(\alpha+n-k) \Gamma(\zeta+n-k)}{\Gamma(k+1) \Gamma(n-k+1) \Gamma(\alpha) \Gamma(\zeta)} \frac{\gamma^{k}\left((-1)^{n-k}+1\right)}{2 \delta^{n-k}} .
\end{align}
Finally, by using basic transformation of the $\Gamma$ function, the n-th moment of the four-parameter SKD can be expressed as
\begin{align} \label{Eq:N_Moment_final}
	\mu_n =\sum_{\substack{k=0 \\ n-k\in\text{even}}}^{n} \frac{n \gamma^{k} \left(\alpha\right)_{n-k} \left(\zeta\right)_{n-k}}{\delta^{n-k} \Gamma(k+1) \left(n\right)_{1-k}} ,
\end{align}
with $\left(x\right)_{n}$ denoting the Pochhammer symbol.

Furthermore, the first four moments and the first four central moments are presented in Table~\ref{Tbl:Moments}.
\begin{table} [htbp]
	\centering
	\caption{Moments and central moments of the SKD.}
	\renewcommand{\arraystretch}{1.5}
	\begin{tabular}{|c|c|} 
		\hline
		\textbf{Moment} & \textbf{Value} \\
		\hline
		$\mu_1$ & $\gamma$ \\
		\hline
		$\mu_2$ & $\frac{\alpha  (\alpha +1) \zeta  (\zeta +1)}{\delta ^2}+\gamma ^2$ \\
		\hline
		$\mu_3$ & $\frac{3 \alpha  (\alpha +1) \gamma  \zeta  (\zeta +1)}{\delta ^2}+\gamma ^3$ \\
		\hline
		$\mu_4$ & $\frac{6 \alpha  (\alpha +1) \gamma ^2 \zeta  (\zeta +1)}{\delta ^2}+\frac{\alpha  (\alpha +1) (\alpha +2) (\alpha +3) \zeta  (\zeta +1) (\zeta +2) (\zeta +3)}{\delta ^4}+\gamma ^4$ \\
		\hline
		\hline
		\textbf{Central moment} & \textbf{Value} \\
		\hline
		$\tilde{\mu}_1 = \mu_1$  & $\gamma$ \\
		\hline
		$\tilde{\mu}_2 = \mu_2 - \mu_1^2$ & $\frac{\alpha  (\alpha +1) \zeta  (\zeta +1)}{\delta ^2}$ \\
		\hline
		$\tilde{\mu}_3 = \mu_3 - 3 \mu_1 \mu_2 + 2 \mu_1^3$ & $0$ \\
		\hline
		$\tilde{\mu}_4 = \mu_4 - 4 \mu_1 \mu_3 + 6 \mu_1^2 \mu_2 - 3 \mu_1^4$ & $\frac{\alpha  (\alpha +1) (\alpha +2) (\alpha +3) \zeta  (\zeta +1) (\zeta +2) (\zeta +3)}{\delta ^4}$ \\
		\hline
	\end{tabular}
	\label{Tbl:Moments}
\end{table}

\subsection{Cumulants}
In addition, the cumulants of the SKD are given by
\begin{align}
	K\left(t\right) = \log\left(\mathrm{E}\left[e^{t x}\right]\right) ,
\end{align}
or, in terms of the N-th moments
\begin{align}
	\kappa_n = \mu_n - \sum_{m=1}^{n-1} \left(\begin{array}{c}n-1 \\ m\end{array}\right) \kappa_m \mu_{n-m} ,
\end{align}
which, after substituting~\eqref{Eq:N_Moment_final}, can be written as
\begin{align}
	\begin{split}
		\kappa_n = \sum_{\substack{k=0 \\ n-k\in\text{even}}}^{n} &\frac{n \left(\alpha\right)_{n-k} \left(\zeta\right)_{n-k} \left(1-n\right)_{k}}{k \left(n-k\right) \left(-1\right)^{k} \Gamma\left(k\right)} - \sum_{m=1}^{n-1} \!\!\!\!\!\! \sum_{\substack{k=0 \\ n-m-k\in\text{even}}}^{n-m}\!\!\!\!\!\!  \left(\begin{array}{c}n-1 \\ m\end{array}\right) \kappa_m \\
		&\times \frac{\left(n-m\right) \left(\alpha\right)_{n-m-k} \left(\zeta\right)_{n-m-k} \left(1-n-m\right)_{k}}{k \left(n-m-k\right) \left(-1\right)^{k} \Gamma\left(k\right)} .
	\end{split}
\end{align}
Thus, the first four cumulants are presented in Table~\ref{Tbl:Cumulants}.
\begin{table} [htbp]
	\centering
	\caption{Cumulants of the SKD.}
	\renewcommand{\arraystretch}{1.5}
	\begin{tabular}{|c|c|} 
		\hline
		\textbf{Cumulant} & \textbf{Value} \\
		\hline
		$\kappa_1 = \tilde{\mu}_1$ & $\gamma$ \\
		\hline
		$\kappa_2 = \tilde{\mu}_2$ & $\frac{\alpha  (\alpha +1) \zeta  (\zeta +1)}{\delta ^2}$ \\
		\hline
		$\kappa_3 = \tilde{\mu}_3$ & $0$ \\
		\hline
		$\kappa_4 = \tilde{\mu}_4 - 3 \tilde{\mu}_2^2$ & $\frac{(\alpha +2) (\alpha +3) (\zeta +2) (\zeta +3)}{\alpha  (\alpha +1) \zeta  (\zeta +1)} \left(\alpha ^2 ((\zeta -1) \zeta -3) - \alpha (\zeta  (\zeta +11)+15)-3 (\zeta +2) (\zeta +3)\right)$ \\
		\hline
	\end{tabular}
	\label{Tbl:Cumulants}
\end{table}

Finally, based on the aforementioned analysis, the mean, variance, skewness, and kurtosis of the SKD are given in
\begin{align}
	\mu = &\mu_1 = \gamma ,\\
	\sigma^2 = &\tilde{\mu}_2 = \frac{\alpha  (\alpha +1) \zeta  (\zeta +1)}{\delta ^2} ,\\
	\gamma_1 = &\frac{\tilde{\mu}_3}{\tilde{\mu}_2^{3/2}} = 0 , \text{ and}\\
	\beta_2 = &\frac{\tilde{\mu}_4}{\tilde{\mu}_2^2} = \frac{(\alpha +2) (\alpha +3) (\zeta +2) (\zeta +3)}{\alpha  (\alpha +1) \zeta  (\zeta +1)} .
\end{align}

\subsection{Order Statistics}
Order statistics are among the most fundamental tools in non-parametric statistics and inference. The order statistics $X_{(1)},X_{(2)},\cdots,X_{(n)}$,  of any random sample of random variables, $X_1,X_2,\dots,X_n$, are the same random variables sorted in increasing order. Assuming that any such random sample follows the SKD, the CDF of that random sample is given by
\begin{align}
	F_{\mathrm{X}_{\left(r\right)}}\left(x\right) = \sum_{j-r}^{n} \left(\begin{array}{c}n \\ j\end{array}\right) \left[F\left(x\right)\right]^{j} \left[1-F\left(x\right)\right]^{n-j} ,
\end{align}
and, after substituting~\eqref{Eq:CDF_final}, it can be equivalently written as
\begin{align}
	F_{\mathrm{X}_{\left(r\right)}}\left(x\right) = \sum_{j-r}^{n} \left(\begin{array}{c}n \\ j\end{array}\right) \left[\frac{1}{2}+\Lambda\left(x\right)\right]^{j} \left[\frac{1}{2}-\Lambda\left(x\right)\right]^{n-j} ,
\end{align}
with
\begin{align}
	\Lambda\left(x\right) = \frac{\mathrm{sgn}\left(x-\gamma\right) \delta^{\frac{\alpha + \zeta}{2}}}{\Gamma(\alpha) \Gamma(\zeta)}|x-\gamma|^{\frac{\alpha + \zeta}{2}} \quad G_{1,3}^{2,1} \left[\delta |x-\gamma| \ \vline \ \begin{array}{cc}1-\frac{\alpha+\zeta}{2}\\ \frac{\alpha-\zeta}{2},-\frac{\alpha-\zeta}{2},-\frac{\alpha+\zeta}{2}\end{array}\right] .
\end{align}

Furthermore, the PDF of a random sample can be expressed as
\begin{align}
	f_{\mathrm{X}_{\left(r\right)}}\left(x\right) = \frac{n!}{\left(r-1\right)! \left(n-r\right)!} f\left(x\right) \left[F\left(x\right)\right]^{r-1} \left[1-F\left(x\right)\right]^{n-r} ,
\end{align}
and, after substituting~\eqref{Eq:PDF_2} and~\eqref{Eq:CDF_final}, it can be rewritten as
\begin{align}
	\begin{split}
		f_{\mathrm{X}_{\left(r\right)}}\left(x\right) =	&\frac{n!}{ \left(r-1\right)! \left(n-r\right)!} \frac{\mathrm{sgn}\left(x-\gamma\right) \delta^{\frac{\alpha + \zeta}{2}}}{\Gamma\left(\alpha\right) \Gamma\left(\zeta\right)} |x-\gamma|^{\frac{\alpha + \zeta}{2} - 1} \\
		&\times K_{\alpha - \zeta} \left(2\sqrt{\delta |x-\gamma|}\right) \left[\frac{1}{2}+\Lambda\left(x\right)\right]^{r-1} \left[\frac{1}{2}-\Lambda\left(x\right)\right]^{n-r}
	\end{split} .
\end{align}

\subsection{Distribution of the product and ratio of two SKD RVs}
In this section we derive the PDF and CDF of the product and ratio distributions of the SKD. These two types of distributions are extensively applied in machine learning and Bayesian analysis problems, such as posterior distribution and density estimation.

\subsubsection{Product PDF}
\begin{theorem} \label{Th:Product_PDF}
	The PDF of the product of two zero mean iid variables that follow the SKD is given by
	\begin{align} \label{Eq:product_pdf_final}
		f_\mathrm{Z}\left(z\right) = \frac{\delta^{\alpha+\zeta} \left|z\right|^{\frac{\alpha+\zeta}{2}-1}}{2\left[\Gamma\left(\alpha\right)\Gamma\left(\zeta\right)\right]^2} G_{0,4}^{4,0} \left(\delta \left|z\right|\ \vline \ \begin{array}{cc}-\\ \frac{\alpha-\zeta}{2},-\frac{\alpha-\zeta}{2},\frac{\alpha-\zeta}{2},-\frac{\alpha-\zeta}{2}\end{array}\right) .
	\end{align}
\end{theorem}
\begin{proof}
	The pdf of the product of two iid variables that follow the symmetric can be expressed as
	\begin{align} \label{Eq:product_pdf_1}
		f_\mathrm{Z}\left(z\right) = \int_{-\infty}^{\infty} f_\mathrm{X}\left(x\right) f_\mathrm{Y}\left(\frac{z}{x}\right) \frac{1}{\left|x\right|} dx ,
	\end{align}
	where $z=xy$. Next, after substituting~\eqref{Eq:PDF_2} into~\eqref{Eq:product_pdf_1}, the latter can be rewritten as
	\begin{align} \label{Eq:product_pdf_2}
		\begin{split}
			f_\mathrm{Z}\left(z\right) = \frac{\delta^{\alpha+\zeta}}{\left[\Gamma\left(\alpha\right)\Gamma\left(\zeta\right)\right]^2} &\int_{-\infty}^{\infty} \frac{\left(\left|x-\gamma\right|\left|\frac{z}{x}-\gamma\right|\right)^{\frac{\alpha+\zeta}{2}-1}}{\left|x\right|} \\
			& \times K_{\alpha-\zeta}\left(2\sqrt{\delta\left|x-\gamma\right|}\right) K_{\alpha-\zeta}\left(2\sqrt{\delta\left|\frac{z}{x}-\gamma\right|}\right) dx ,
		\end{split}
	\end{align}
	Furthermore, by expressing the $K_v$ function in terms of the Meijer-G function as in~\cite[eq. 8.4.23.1]{B:PrudnIII}, the previous equation can be equivalently written as
	\begin{align} \label{Eq:product_pdf_3}
		\begin{split}
			f_\mathrm{Z}\left(z\right) = &\frac{\delta^{\alpha+\zeta}}{\left[2\Gamma\left(\alpha\right)\Gamma\left(\zeta\right)\right]^2} \int_{-\infty}^{\infty} \frac{\left(\left|x-\gamma\right|\left|\frac{z}{x}-\gamma\right|\right)^{\frac{\alpha+\zeta}{2}-1}}{\left|x\right|} \\
			& \times G_{0,2}^{2,0} \left[\delta \left|x-\gamma\right| \ \vline \ \begin{array}{cc}-\\ \frac{\alpha-\zeta}{2},-\frac{\alpha-\zeta}{2}\end{array}\right] G_{0,2}^{2,0} \left[\delta \left|\frac{z}{x}-\gamma\right| \ \vline \ \begin{array}{cc}-\\ \frac{\alpha-\zeta}{2},-\frac{\alpha-\zeta}{2}\end{array}\right] dx .
		\end{split}
	\end{align}
	
	It becomes evident that the evaluation of a closed form expression for the integral presented in~\eqref{Eq:product_pdf_3} is very challenging. However, if we assume zero mean symmetric K iid variables,~\eqref{Eq:product_pdf_3} can be expressed as
	\begin{align} \label{Eq:product_pdf_4}
		\begin{split}
			f_\mathrm{Z}\left(z\right) = &\frac{\delta^{\alpha+\zeta} \left|z\right|^{\frac{\alpha+\zeta}{2}-1}}{\left[2\Gamma\left(\alpha\right)\Gamma\left(\zeta\right)\right]^2} \int_{-\infty}^{\infty} \frac{1}{\left|x\right|}  G_{0,2}^{2,0} \left[\delta \left|x\right| \ \vline \ \begin{array}{cc}-\\ \frac{\alpha-\zeta}{2},-\frac{\alpha-\zeta}{2}\end{array}\right] G_{0,2}^{2,0} \left[\frac{\delta \left|z\right|}{\left|x\right|} \ \vline \ \begin{array}{cc}-\\ \frac{\alpha-\zeta}{2},-\frac{\alpha-\zeta}{2}\end{array}\right] dx ,
		\end{split}
	\end{align}
	or, equivalently
	\begin{align} \label{Eq:product_pdf_5}
		\begin{split}
			f_\mathrm{Z}\left(z\right) = \frac{\delta^{\alpha+\zeta} \left|z\right|^{\frac{\alpha+\zeta}{2}-1}}{\left[2\Gamma\left(\alpha\right)\Gamma\left(\zeta\right)\right]^2} &\left[\int_{0}^{\infty} \frac{1}{x}  G_{0,2}^{2,0} \left(\delta x \ \vline \ \begin{array}{cc}-\\ \frac{\alpha-\zeta}{2},-\frac{\alpha-\zeta}{2}\end{array}\right) G_{0,2}^{2,0} \left(\frac{\delta \left|z\right|}{x} \ \vline \ \begin{array}{cc}-\\ \frac{\alpha-\zeta}{2},-\frac{\alpha-\zeta}{2}\end{array}\right) dx \right. \\
			&- \left. \int_{-\infty}^{0} \frac{1}{x}  G_{0,2}^{2,0} \left(-\delta x \ \vline \ \begin{array}{cc}-\\ \frac{\alpha-\zeta}{2},-\frac{\alpha-\zeta}{2}\end{array}\right) G_{0,2}^{2,0} \left(- \frac{\delta \left|z\right|}{x} \ \vline \ \begin{array}{cc}-\\ \frac{\alpha-\zeta}{2},-\frac{\alpha-\zeta}{2}\end{array}\right) dx \right] ,
		\end{split}
	\end{align}
	and, after using the substitution $x=-x$ on the second integral,~\eqref{Eq:product_pdf_5} can be rewritten as
	\begin{align} \label{Eq:product_pdf_6}		
		f_\mathrm{Z}\left(z\right) = \frac{\delta^{\alpha+\zeta} \left|z\right|^{\frac{\alpha+\zeta}{2}-1}}{2\left[\Gamma\left(\alpha\right)\Gamma\left(\zeta\right)\right]^2} \int_{0}^{\infty} \frac{1}{x}  G_{0,2}^{2,0} \left(\delta x \ \vline \ \begin{array}{cc}-\\ \frac{\alpha-\zeta}{2},-\frac{\alpha-\zeta}{2}\end{array}\right) G_{0,2}^{2,0} \left(\frac{\delta \left|z\right|}{x} \ \vline \ \begin{array}{cc}-\\ \frac{\alpha-\zeta}{2},-\frac{\alpha-\zeta}{2}\end{array}\right) dx .
	\end{align}
	Next, by using~\cite[eq. 8.2.2.14]{B:PrudnIII}, the previous equation can be expressed as
	\begin{align} \label{Eq:product_pdf_7}
		f_\mathrm{Z}\left(z\right) = \frac{\delta^{\alpha+\zeta} \left|z\right|^{\frac{\alpha+\zeta}{2}-1}}{2\left[\Gamma\left(\alpha\right)\Gamma\left(\zeta\right)\right]^2} \int_{0}^{\infty} \frac{1}{x}  G_{0,2}^{2,0} \left(\delta x \ \vline \ \begin{array}{cc}-\\ \frac{\alpha-\zeta}{2},-\frac{\alpha-\zeta}{2}\end{array}\right) G_{2,0}^{0,2} \left(\frac{x}{\delta \left|z\right|} \ \vline \ \begin{array}{cc}1-\frac{\alpha-\zeta}{2},1+\frac{\alpha-\zeta}{2}\\ -\end{array}\right) dx .
	\end{align}
	Moreover, by performing the integration based on~\cite[eq. 2.24.1.3]{B:PrudnIII}, the previous equation can be rewritten as
	\begin{align} \label{Eq:product_pdf_8}
		f_\mathrm{Z}\left(z\right) = \frac{\delta^{\alpha+\zeta} \left|z\right|^{\frac{\alpha+\zeta}{2}-1}}{2\left[\Gamma\left(\alpha\right)\Gamma\left(\zeta\right)\right]^2} G_{4,0}^{0,4} \left(\frac{1}{\delta \left|z\right|} \ \vline \ \begin{array}{cc}1-\frac{\alpha-\zeta}{2},1+\frac{\alpha-\zeta}{2},1-\frac{\alpha-\zeta}{2},1+\frac{\alpha-\zeta}{2}\\ -\end{array}\right) .
	\end{align}
	Finally, after inverting the argument in~\eqref{Eq:product_pdf_8}, the PDF of the product distribution is presented in~\eqref{Eq:product_pdf_final}. This concludes the proof.
\end{proof}

\subsubsection{Product CDF}
\begin{theorem} \label{Th:Product_CDF}
	The CDF of the product of two zero mean iid variables that follow the SKD is given by
	\begin{align} \label{Eq:product_cdf_final}
		\begin{split}
			F_\mathrm{Z}\left(z\right) = &\frac{\delta^{\alpha+\zeta}}{2\left[\Gamma\left(\alpha\right)\Gamma\left(\zeta\right)\right]^2} \left[\delta^{\frac{\alpha+\zeta}{2}} \left[\Gamma\left(\alpha\right)\Gamma\left(\zeta\right)\right]^2 \right. \\ 
			&\left. + \mathrm{sgn}\left(z\right) \left|z\right|^{\frac{\alpha+\zeta}{2}} G_{1,5}^{4,1} \left(\delta \left|z\right|\ \vline \ \begin{array}{cc} 1-\frac{\alpha+\zeta}{2} \\ \frac{\alpha-\zeta}{2}, -\frac{\alpha-\zeta}{2}, \frac{\alpha-\zeta}{2}, -\frac{\alpha-\zeta}{2}, -\frac{\alpha+\zeta}{2} \end{array}\right) \right] .
		\end{split}
	\end{align}
\end{theorem}
\begin{proof}
	The CDF of the product of two RVs that follow the standard SKD with identical parameters is given by
	\begin{align} \label{Eq:product_cdf_1}
		F_\mathrm{Z}\left(z\right) = \int_{-\infty}^{z} f_\mathrm{Z}\left(x\right)dx ,
	\end{align}
	After substituting~\eqref{Eq:product_pdf_final} into~\eqref{Eq:product_cdf_1}, the later can be rewritten as
	\begin{align} \label{Eq:product_cdf_2}
		F_\mathrm{Z}\left(z\right) = \frac{\delta^{\alpha+\zeta}}{2\left[\Gamma\left(\alpha\right)\Gamma\left(\zeta\right)\right]^2} \int_{-\infty}^{z} \left|x\right|^{\frac{\alpha+\zeta}{2}-1} G_{0,4}^{4,0} \left(\delta \left|x\right|\ \vline \ \begin{array}{cc}-\\ \frac{\alpha-\zeta}{2},-\frac{\alpha-\zeta}{2},\frac{\alpha-\zeta}{2},-\frac{\alpha-\zeta}{2}\end{array}\right) dx .
	\end{align}
	On the first hand, if $z>0$ the previous equation can transformed into 
	\begin{align} \label{Eq:product_cdf_3a}
		\begin{split}
			F_\mathrm{Z}\left(z\right) = &\frac{\delta^{\alpha+\zeta}}{2\left[\Gamma\left(\alpha\right)\Gamma\left(\zeta\right)\right]^2} \left[\int_{-\infty}^{0} \left(-x\right)^{\frac{\alpha+\zeta}{2}-1} G_{0,4}^{4,0} \left(\delta \left(-x\right)\ \vline \ \begin{array}{cc}-\\ \frac{\alpha-\zeta}{2},-\frac{\alpha-\zeta}{2},\frac{\alpha-\zeta}{2},-\frac{\alpha-\zeta}{2}\end{array}\right) dx \right. \\
			&\left.+ \int_{0}^{z} x^{\frac{\alpha+\zeta}{2}-1} G_{0,4}^{4,0} \left(\delta x\ \vline \ \begin{array}{cc}-\\ \frac{\alpha-\zeta}{2},-\frac{\alpha-\zeta}{2},\frac{\alpha-\zeta}{2},-\frac{\alpha-\zeta}{2}\end{array}\right) dx\right] ,
		\end{split}
	\end{align}
	or equivalently
	\begin{align} \label{Eq:product_cdf_4a}
		\begin{split}
			F_\mathrm{Z}\left(z\right) = &\frac{\delta^{\alpha+\zeta}}{2\left[\Gamma\left(\alpha\right)\Gamma\left(\zeta\right)\right]^2} \left[\int_{0}^{\infty} x^{\frac{\alpha+\zeta}{2}-1} G_{0,4}^{4,0} \left(\delta x\ \vline \ \begin{array}{cc}-\\ \frac{\alpha-\zeta}{2},-\frac{\alpha-\zeta}{2},\frac{\alpha-\zeta}{2},-\frac{\alpha-\zeta}{2}\end{array}\right) dx \right. \\
			&\left.+ \int_{0}^{z} x^{\frac{\alpha+\zeta}{2}-1} G_{0,4}^{4,0} \left(\delta x\ \vline \ \begin{array}{cc}-\\ \frac{\alpha-\zeta}{2},-\frac{\alpha-\zeta}{2},\frac{\alpha-\zeta}{2},-\frac{\alpha-\zeta}{2}\end{array}\right) dx\right] .
		\end{split}
	\end{align}
	Furthermore, we evaluate the integrals in~\eqref{Eq:product_cdf_4a} by using~\cite[eq. 24, 26]{Adamchik1990} and, thus, it can be rewritten as
	\begin{align}
		\begin{split}
			F_\mathrm{Z}\left(z\right) = &\frac{\delta^{\alpha+\zeta}}{2\left[\Gamma\left(\alpha\right)\Gamma\left(\zeta\right)\right]^2} \left[\delta^{\frac{\alpha+\zeta}{2}} \left[\Gamma\left(\alpha\right)\Gamma\left(\zeta\right)\right]^2 \right. \\ 
			&\left. + z^{\frac{\alpha+\zeta}{2}} G_{1,5}^{4,1} \left(\delta z\ \vline \ \begin{array}{cc} 1-\frac{\alpha+\zeta}{2} \\ \frac{\alpha-\zeta}{2}, -\frac{\alpha-\zeta}{2}, \frac{\alpha-\zeta}{2}, -\frac{\alpha-\zeta}{2}, -\frac{\alpha+\zeta}{2} \end{array}\right) \right] .
		\end{split}
	\end{align}
	On the other hand, if $z<0$ the previous equation can transformed into 
	\begin{align} \label{Eq:product_cdf_3b}
		F_\mathrm{Z}\left(z\right) = \frac{\delta^{\alpha+\zeta}}{2\left[\Gamma\left(\alpha\right)\Gamma\left(\zeta\right)\right]^2} \int_{-\infty}^{z} \left(-x\right)^{\frac{\alpha+\zeta}{2}-1} G_{0,4}^{4,0} \left(\delta \left(-x\right)\ \vline \ \begin{array}{cc}-\\ \frac{\alpha-\zeta}{2},-\frac{\alpha-\zeta}{2},\frac{\alpha-\zeta}{2},-\frac{\alpha-\zeta}{2}\end{array}\right) dx ,
	\end{align}
	which can be transformed into
	\begin{align} \label{Eq:product_cdf_4b}
		F_\mathrm{Z}\left(z\right) = \frac{\delta^{\alpha+\zeta}}{2\left[\Gamma\left(\alpha\right)\Gamma\left(\zeta\right)\right]^2} \int_{-z}^{\infty} x^{\frac{\alpha+\zeta}{2}-1} G_{0,4}^{4,0} \left(\delta x\ \vline \ \begin{array}{cc}-\\ \frac{\alpha-\zeta}{2},-\frac{\alpha-\zeta}{2},\frac{\alpha-\zeta}{2},-\frac{\alpha-\zeta}{2}\end{array}\right) dx ,
	\end{align}
	or equivalently
	\begin{align} \label{Eq:product_cdf_5b}
		\begin{split}
			F_\mathrm{Z}\left(z\right) = &\frac{\delta^{\alpha+\zeta}}{2\left[\Gamma\left(\alpha\right)\Gamma\left(\zeta\right)\right]^2} \left[\int_{0}^{\infty} x^{\frac{\alpha+\zeta}{2}-1} G_{0,4}^{4,0} \left(\delta x\ \vline \ \begin{array}{cc}-\\ \frac{\alpha-\zeta}{2},-\frac{\alpha-\zeta}{2},\frac{\alpha-\zeta}{2},-\frac{\alpha-\zeta}{2}\end{array}\right) dx \right. \\
			&\left.- \int_{0}^{-z} x^{\frac{\alpha+\zeta}{2}-1} G_{0,4}^{4,0} \left(\delta x\ \vline \ \begin{array}{cc}-\\ \frac{\alpha-\zeta}{2},-\frac{\alpha-\zeta}{2},\frac{\alpha-\zeta}{2},-\frac{\alpha-\zeta}{2}\end{array}\right) dx\right] .
		\end{split}
	\end{align}
	Furthermore, we evaluate the integrals in~\eqref{Eq:product_cdf_5b} by using~\cite[eq. 24, 26]{Adamchik1990} and, thus, it can be rewritten as
	\begin{align}
		\begin{split}
			F_\mathrm{Z}\left(z\right) = &\frac{\delta^{\alpha+\zeta}}{2\left[\Gamma\left(\alpha\right)\Gamma\left(\zeta\right)\right]^2} \left[\delta^{\frac{\alpha+\zeta}{2}} \left[\Gamma\left(\alpha\right)\Gamma\left(\zeta\right)\right]^2 \right. \\ 
			&\left. - \left(-z\right)^{\frac{\alpha+\zeta}{2}} G_{1,5}^{4,1} \left(\delta \left(-z\right)\ \vline \ \begin{array}{cc} 1-\frac{\alpha+\zeta}{2} \\ \frac{\alpha-\zeta}{2}, -\frac{\alpha-\zeta}{2}, \frac{\alpha-\zeta}{2}, -\frac{\alpha-\zeta}{2}, -\frac{\alpha+\zeta}{2} \end{array}\right) \right] .
		\end{split}
	\end{align}
	Finally, by combining the two branches presented in~\eqref{Eq:product_cdf_4a} and~\eqref{Eq:product_cdf_5b}, the CDF of the product distribution is presented in~\eqref{Eq:product_cdf_final}. This concludes the proof.
\end{proof}

\subsubsection{Ratio PDF}
\begin{theorem} \label{Th:Ratio_PDF}
	The PDF of the ratio of two zero mean iid variables that follow the SKD is given by
	\begin{align} \label{Eq:ratio_pdf_final3}
		f_\mathrm{Z}\left(z\right) = \frac{\delta^{\alpha+\zeta}}{\left|z\right|^{1+\alpha}}  \frac{2 \alpha \zeta \left(\alpha+\zeta\right)^2 B\left(2\alpha,2\zeta\right)}{\left(2\alpha+2\zeta\right)_{2} \left[B\left(\alpha,\zeta\right)\right]^2}\ _{2}\!F_{1}\!\!\left(\!\!1+\alpha+\zeta, 1+2\alpha ; 2+2\alpha+2\zeta ; 1\!-\!\frac{1}{\left|z\right|}\right).
	\end{align}		
\end{theorem}
\begin{proof}
	The pdf of the ratio of two iid variables that follow the SKD can be expressed as
	\begin{align} \label{Eq:ratio_pdf_1}
		f_\mathrm{Z}\left(z\right) = \int_{-\infty}^{\infty} f_\mathrm{X}\left(z y\right) f_\mathrm{Y}\left(y\right) \left|y\right| dy ,
	\end{align}
	where $z=\frac{x}{y}$. Next, after substituting~\eqref{Eq:PDF_2} into~\eqref{Eq:ratio_pdf_1}, the latter can be rewritten as
	\begin{align} \label{Eq:ratio_pdf_2}
		\begin{split}
			f_\mathrm{Z}\left(z\right) = \frac{\delta^{\alpha+\zeta}}{\left[\Gamma\left(\alpha\right)\Gamma\left(\zeta\right)\right]^2} &\int_{-\infty}^{\infty} \left|y\right| \left(\left|y-\gamma\right|\left|z y-\gamma\right|\right)^{\frac{\alpha+\zeta}{2}-1} \\
			& \times K_{\alpha-\zeta}\left(2\sqrt{\delta\left|y-\gamma\right|}\right) K_{\alpha-\zeta}\left(2\sqrt{\delta\left|z y-\gamma\right|}\right) dy ,
		\end{split}
	\end{align}
	Furthermore, by expressing the $K_v$ function in terms of the Meijer-G function as in~\cite[eq. 8.4.23.1]{B:PrudnIII}, the previous equation can be equivalently written as
	\begin{align} \label{Eq:ratio_pdf_3}
		\begin{split}
			f_\mathrm{Z}\left(z\right) = &\frac{\delta^{\alpha+\zeta}}{\left[2\Gamma\left(\alpha\right)\Gamma\left(\zeta\right)\right]^2} \int_{-\infty}^{\infty} \left|y\right| \left(\left|y-\gamma\right|\left|z y-\gamma\right|\right)^{\frac{\alpha+\zeta}{2}-1} \\
			& \times G_{0,2}^{2,0} \left[\delta \left|y-\gamma\right| \ \vline \ \begin{array}{cc}-\\ \frac{\alpha-\zeta}{2},-\frac{\alpha-\zeta}{2}\end{array}\right] G_{0,2}^{2,0} \left[\delta \left|z y-\gamma\right| \ \vline \ \begin{array}{cc}-\\ \frac{\alpha-\zeta}{2},-\frac{\alpha-\zeta}{2}\end{array}\right] dy .
		\end{split}
	\end{align}
	Next, we assume zero mean symmetric K iid variables and~\eqref{Eq:ratio_pdf_3} can be expressed as
	\begin{align} \label{Eq:ratio_pdf_4}
		\begin{split}
			f_\mathrm{Z}\left(z\right) = &\frac{\delta^{\alpha+\zeta} \left|z\right|^{\frac{\alpha+\zeta}{2}-1}}{\left[2\Gamma\left(\alpha\right)\Gamma\left(\zeta\right)\right]^2} \left[\int_{0}^{\infty} y^{\alpha+\zeta-1}  G_{0,2}^{2,0} \left(\delta y \ \vline \ \begin{array}{cc}-\\ \frac{\alpha-\zeta}{2},-\frac{\alpha-\zeta}{2}\end{array}\right) G_{0,2}^{2,0} \left(\delta \left|z\right| y \ \vline \ \begin{array}{cc}-\\ \frac{\alpha-\zeta}{2},-\frac{\alpha-\zeta}{2}\end{array}\right) dy \right. \\
			&+ \left. \int_{-\infty}^{0}  \left(-y\right)^{\alpha+\zeta-1}  G_{0,2}^{2,0} \left(-\delta y \ \vline \ \begin{array}{cc}-\\ \frac{\alpha-\zeta}{2},-\frac{\alpha-\zeta}{2}\end{array}\right) G_{0,2}^{2,0} \left(- \delta \left|z\right| y \ \vline \ \begin{array}{cc}-\\ \frac{\alpha-\zeta}{2},-\frac{\alpha-\zeta}{2}\end{array}\right) dy \right] ,
		\end{split}
	\end{align}
	which, after using the substitution $y=-y$ on the second integral can be rewritten as
	\begin{align} \label{Eq:ratio_pdf_5}
		f_\mathrm{Z}\left(z\right) = \frac{\delta^{\alpha+\zeta} \left|z\right|^{\frac{\alpha+\zeta}{2}-1}}{2\left[\Gamma\left(\alpha\right)\Gamma\left(\zeta\right)\right]^2} \int_{0}^{\infty} y^{\alpha+\zeta-1}  G_{0,2}^{2,0} \left(\delta y \ \vline \ \begin{array}{cc}-\\ \frac{\alpha-\zeta}{2},-\frac{\alpha-\zeta}{2}\end{array}\right) G_{0,2}^{2,0} \left(\delta \left|z\right| y \ \vline \ \begin{array}{cc}-\\ \frac{\alpha-\zeta}{2},-\frac{\alpha-\zeta}{2}\end{array}\right) dy ,
	\end{align}
	Finally, since $\arg \delta < \pi$ and $\arg \delta\left|z\right| < \pi$, we can perform the integration based on~\cite[eq. 2.24.1.3]{B:PrudnIII}. Thus, the PDF of the ratio distribution can be expressed as
	\begin{align} \label{Eq:ratio_pdf_final}
		f_\mathrm{Z}\left(z\right) = \frac{\delta^{\alpha+\zeta} \left|z\right|^{\frac{\alpha+\zeta}{2}-1}}{2\left[\Gamma\left(\alpha\right)\Gamma\left(\zeta\right)\right]^2} G_{2,2}^{2,2} \left( \delta \left|z\right| \ \vline \ \begin{array}{cc}1-\frac{3\alpha+\zeta}{2},1-\frac{\alpha+3\zeta}{2}\\ 1-\frac{\alpha-\zeta}{2},1+\frac{\alpha-\zeta}{2}\end{array}\right) .
	\end{align}
	Furthermore, after expressing Meijer-G in terms of the $_{2}\!F_{1}$ hypergeometric function based on~\cite[eq. 9.34.7]{Gradshteyn2014}, can be written as
	\begin{align} \label{Eq:ratio_pdf_final2}
		\begin{split}
			f_\mathrm{Z}\left(z\right) = &\frac{\delta^{\alpha+\zeta}}{\left|z\right|^{1+\alpha}} \frac{\left[\Gamma\left(1+\alpha+\zeta\right)\right]^2 \Gamma\left(1+2\alpha\right) \Gamma\left(1+2\zeta\right)}{2\left[\Gamma\left(\alpha\right)\Gamma\left(\zeta\right)\right]^2 \Gamma\left(2+2\alpha+2\zeta\right)} \\
			&\times\ _{2}\!F_{1}\left(1+\alpha+\zeta, 1+2\alpha ; 2+2\alpha+2\zeta ; 1-\frac{1}{\left|z\right|}\right).
		\end{split}
	\end{align}
	and, after performing basic transformation of the Gamma function, it can be transformed as shown in~\eqref{Eq:ratio_pdf_final3}. This concludes the proof.
\end{proof}

\subsubsection{Ratio CDF}
\begin{theorem} \label{Th:Ratio_CDF}
	The CDF of the ratio of two zero mean iid variables that follow the SKD is given by
	\begin{align} \label{Eq:ratio_cdf_final}
		\begin{split}
			F_\mathrm{Z}\left(z\right) = &\frac{\delta^{\alpha+\zeta}}{2\left[\Gamma\left(\alpha\right)\Gamma\left(\zeta\right)\right]^2} \left[\delta^{\frac{\alpha+\zeta}{2}} \Gamma\left(1+\alpha\right) \Gamma\left(1+\zeta\right) \Gamma\left(\alpha\right) \Gamma\left(\zeta\right) \right. \\ 
			&\left. + \mathrm{sgn}\left(z\right) \left|z\right|^{\frac{\alpha+\zeta}{2}} G_{3,3}^{2,3} \left(\delta \left|z\right|\ \vline \ \begin{array}{cc} 1-\frac{3\alpha+\zeta}{2}, 1-\frac{\alpha+3\zeta}{2}, 1-\frac{\alpha+\zeta}{2} \\ 1-\frac{\alpha-\zeta}{2}, 1+\frac{\alpha-\zeta}{2}, -\frac{\alpha+\zeta}{2} \end{array}\right) \right] .
		\end{split}
	\end{align}
\end{theorem}
\begin{proof}
	The cdf of the ratio of two RVs that follow the standard SKD with identical parameters is given by
	\begin{align} \label{Eq:ratio_cdf_1}
		F_\mathrm{Z}\left(z\right) = \int_{-\infty}^{z} f_\mathrm{Z}\left(x\right)dx ,
	\end{align}
	After substituting~\eqref{Eq:ratio_pdf_final} into~\eqref{Eq:ratio_cdf_1}, the later can be rewritten as
	\begin{align} \label{Eq:ratio_cdf_2}
		F_\mathrm{Z}\left(z\right) = \frac{\delta^{\alpha+\zeta}}{2\left[\Gamma\left(\alpha\right)\Gamma\left(\zeta\right)\right]^2} \int_{-\infty}^{z} \left|x\right|^{\frac{\alpha+\zeta}{2}-1} G_{2,2}^{2,2} \left(\delta \left|x\right|\ \vline \ \begin{array}{cc} 1-\frac{3\alpha+\zeta}{2}, 1-\frac{\alpha+3\zeta}{2} \\ 1-\frac{\alpha-\zeta}{2}, 1+\frac{\alpha-\zeta}{2}\end{array}\right) dx .
	\end{align}
	On the first hand, if $z>0$ the previous equation can transformed into 
	\begin{align} \label{Eq:ratio_cdf_3a}
		\begin{split}
			F_\mathrm{Z}\left(z\right) = &\frac{\delta^{\alpha+\zeta}}{2\left[\Gamma\left(\alpha\right)\Gamma\left(\zeta\right)\right]^2} \left[\int_{-\infty}^{0} \left(-x\right)^{\frac{\alpha+\zeta}{2}-1} G_{2,2}^{2,2} \left(\delta \left(-x\right)\ \vline \ \begin{array}{cc} 1-\frac{3\alpha+\zeta}{2}, 1-\frac{\alpha+3\zeta}{2} \\ 1-\frac{\alpha-\zeta}{2}, 1+\frac{\alpha-\zeta}{2}\end{array}\right) dx \right. \\
			&\left.+ \int_{0}^{z} x^{\frac{\alpha+\zeta}{2}-1} G_{2,2}^{2,2} \left(\delta x\ \vline \ \begin{array}{cc} 1-\frac{3\alpha+\zeta}{2}, 1-\frac{\alpha+3\zeta}{2} \\ 1-\frac{\alpha-\zeta}{2}, 1+\frac{\alpha-\zeta}{2}\end{array}\right) dx\right] ,
		\end{split}
	\end{align}
	or equivalently
	\begin{align} \label{Eq:ratio_cdf_4a}
		\begin{split}
			F_\mathrm{Z}\left(z\right) = &\frac{\delta^{\alpha+\zeta}}{2\left[\Gamma\left(\alpha\right)\Gamma\left(\zeta\right)\right]^2} \left[\int_{0}^{\infty} x^{\frac{\alpha+\zeta}{2}-1} G_{2,2}^{2,2} \left(\delta x\ \vline \ \begin{array}{cc} 1-\frac{3\alpha+\zeta}{2}, 1-\frac{\alpha+3\zeta}{2} \\ 1-\frac{\alpha-\zeta}{2}, 1+\frac{\alpha-\zeta}{2}\end{array}\right) dx \right. \\
			&\left.+ \int_{0}^{z} x^{\frac{\alpha+\zeta}{2}-1} G_{2,2}^{2,2} \left(\delta x\ \vline \ \begin{array}{cc} 1-\frac{3\alpha+\zeta}{2}, 1-\frac{\alpha+3\zeta}{2} \\ 1-\frac{\alpha-\zeta}{2}, 1+\frac{\alpha-\zeta}{2}\end{array}\right) dx\right] .
		\end{split}
	\end{align}
	Furthermore, we evaluate the integrals in~\eqref{Eq:ratio_cdf_4a} by using~\cite[eq. 24, 26]{Adamchik1990} and, thus, it can be rewritten as
	\begin{align} \label{Eq:ratio_cdf_5a}
		\begin{split}
			F_\mathrm{Z}\left(z\right) = &\frac{\delta^{\alpha+\zeta}}{2\left[\Gamma\left(\alpha\right)\Gamma\left(\zeta\right)\right]^2} \left[\delta^{\frac{\alpha+\zeta}{2}} \Gamma\left(1+\alpha\right) \Gamma\left(1+\zeta\right) \Gamma\left(\alpha\right) \Gamma\left(\zeta\right) \right. \\ 
			&\left. + z^{\frac{\alpha+\zeta}{2}} G_{3,3}^{2,3} \left(\delta z\ \vline \ \begin{array}{cc} 1-\frac{3\alpha+\zeta}{2}, 1-\frac{\alpha+3\zeta}{2}, 1-\frac{\alpha+\zeta}{2} \\ 1-\frac{\alpha-\zeta}{2}, 1+\frac{\alpha-\zeta}{2}, -\frac{\alpha+\zeta}{2} \end{array}\right) \right] .
		\end{split}
	\end{align}
	On the other hand, if $z<0$ the previous equation can transformed into 
	\begin{align} \label{Eq:ratio_cdf_3b}
		F_\mathrm{Z}\left(z\right) = \frac{\delta^{\alpha+\zeta}}{2\left[\Gamma\left(\alpha\right)\Gamma\left(\zeta\right)\right]^2} \int_{-\infty}^{z} \left(-x\right)^{\frac{\alpha+\zeta}{2}-1} G_{2,2}^{2,2} \left(\delta \left(-x\right)\ \vline \ \begin{array}{cc} 1-\frac{3\alpha+\zeta}{2}, 1-\frac{\alpha+3\zeta}{2} \\ 1-\frac{\alpha-\zeta}{2}, 1+\frac{\alpha-\zeta}{2}\end{array} \right) dx ,
	\end{align}
	which can be transformed into
	\begin{align} \label{Eq:ratio_cdf_4b}
		F_\mathrm{Z}\left(z\right) = \frac{\delta^{\alpha+\zeta}}{2\left[\Gamma\left(\alpha\right)\Gamma\left(\zeta\right)\right]^2} \int_{-z}^{\infty} x^{\frac{\alpha+\zeta}{2}-1} G_{2,2}^{2,2} \left(\delta x\ \vline \ \begin{array}{cc} 1-\frac{3\alpha+\zeta}{2}, 1-\frac{\alpha+3\zeta}{2} \\ 1-\frac{\alpha-\zeta}{2}, 1+\frac{\alpha-\zeta}{2}\end{array} \right) dx ,
	\end{align}
	or equivalently
	\begin{align} \label{Eq:ratio_cdf_5b}
		\begin{split}
			F_\mathrm{Z}\left(z\right) = &\frac{\delta^{\alpha+\zeta}}{2\left[\Gamma\left(\alpha\right)\Gamma\left(\zeta\right)\right]^2} \left[\int_{0}^{\infty} x^{\frac{\alpha+\zeta}{2}-1} G_{2,2}^{2,2} \left(\delta x\ \vline \ \begin{array}{cc} 1-\frac{3\alpha+\zeta}{2}, 1-\frac{\alpha+3\zeta}{2} \\ 1-\frac{\alpha-\zeta}{2}, 1+\frac{\alpha-\zeta}{2}\end{array} \right) dx \right. \\
			&\left.- \int_{0}^{-z} x^{\frac{\alpha+\zeta}{2}-1} G_{2,2}^{2,2} \left(\delta x\ \vline \ \begin{array}{cc} 1-\frac{3\alpha+\zeta}{2}, 1-\frac{\alpha+3\zeta}{2} \\ 1-\frac{\alpha-\zeta}{2}, 1+\frac{\alpha-\zeta}{2}\end{array} \right) dx\right] .
		\end{split}
	\end{align}
	Furthermore, we evaluate the integrals in~\eqref{Eq:ratio_cdf_5b} by using~\cite[eq. 24, 26]{Adamchik1990} and, thus, it can be rewritten as
	\begin{align} \label{Eq:ratio_cdf_6b}
		\begin{split}
			F_\mathrm{Z}\left(z\right) = &\frac{\delta^{\alpha+\zeta}}{2\left[\Gamma\left(\alpha\right)\Gamma\left(\zeta\right)\right]^2} \left[\delta^{\frac{\alpha+\zeta}{2}} \Gamma\left(1+\alpha\right) \Gamma\left(1+\zeta\right) \Gamma\left(\alpha\right) \Gamma\left(\zeta\right) \right. \\ 
			&\left. - \left(-z\right)^{\frac{\alpha+\zeta}{2}} G_{3,3}^{2,3} \left(\delta \left(-z\right)\ \vline \ \begin{array}{cc} 1-\frac{3\alpha+\zeta}{2}, 1-\frac{\alpha+3\zeta}{2}, 1-\frac{\alpha+\zeta}{2} \\ 1-\frac{\alpha-\zeta}{2}, 1+\frac{\alpha-\zeta}{2}, -\frac{\alpha+\zeta}{2} \end{array}\right) \right] .
		\end{split}
	\end{align}
	Finally, by combining the two branches presented in~\eqref{Eq:ratio_cdf_5a} and~\eqref{Eq:ratio_cdf_6b}, the CDF of the product distribution is presented in~\eqref{Eq:ratio_cdf_final}. This concludes the proof.
\end{proof}

\section{The skew-symmetric K-distribution} \label{S:skew-SKD}
The skew-SKD is a continuous probability distribution that generalizes the SKD to enable non-zero skewness values. To achieve this, an extra parameter, $\lambda$, that regulated the skewness of the distribution is introduced. The PDF of the skew-SKD can be obtained by 
\begin{align} \label{Eq:skew_PDF_1}
	\phi\left(x\right) = 2 f\left(x\right) F\left(\lambda x\right) ,
\end{align}
which, after using~\eqref{Eq:PDF_2} and~\eqref{Eq:CDF_final}, can be rewritten as
\begin{align} \label{Eq:skew_PDF_2}
	\begin{split}
		\phi\left(x\right) = &\frac{\delta^{\frac{\alpha + \zeta}{2}}}{\Gamma\left(\alpha\right) \Gamma\left(\zeta\right)} \left|x - \gamma\right|^{\frac{\alpha + \zeta}{2} - 1} K_{\alpha - \zeta} \left(2\sqrt{\delta \left|x - \gamma\right|}\right) \Bigg[1 \Bigg.\\
		&\left. + \frac{\mathrm{sgn}\left(\lambda x-\gamma\right) \delta^{\frac{\alpha + \zeta}{2}}}{\Gamma(\alpha) \Gamma(\zeta)}|\lambda x-\gamma|^{\frac{\alpha + \zeta}{2}} G_{1,3}^{2,1} \left(\delta |\lambda x-\gamma| \ \vline \ \begin{array}{cc}1-\frac{\alpha+\zeta}{2}\\ \frac{\alpha-\zeta}{2},-\frac{\alpha-\zeta}{2},-\frac{\alpha+\zeta}{2}\end{array}\right)\right] .
	\end{split}
\end{align}

\section{Conclusions} \label{S:Conclusions}
Our work on the SKD and skew-SKD was motivated by our interest on machine learning and, in particular, on Bayesian inference models. In more detail, the promising potential of distributions with adaptable tail nature in Bayesian learning applications, such as sparse signal reconstruction, incentivised our efforts towards finding more general distribution families that can more accurately model such complex dynamics. Specifically, we introduced a four-parameter distribution that is derived as a mixture of the parental three-parameter reflected Gamma distribution by using the two-parameter Gamma as prior. The proposed distribution is termed Symmetric K distribution due to its dependence on the K-function as well as its similarities to the K-distribution. Closed-form expressions of the basic metrics of the SKD are derived, namely PDF, CDF, moments, cumulants, order statistics, as well as product and ratio distributions. In addition, the skew-SKD is calculated in order to enable non-zero skewness values. The derived distributions exhibit great promise for applications in various fields, including machine learning, Bayesian learning, communications and econometric models.

\bibliographystyle{Chicago}
\bibliography{bibliography}

\begin{thebibliography}{}

\bibitem[\protect\citeauthoryear{Adamchik and Marichev}{Adamchik and
  Marichev}{1990}]{Adamchik1990}
Adamchik, V.~S. and O.~I. Marichev (1990).
\newblock The algorithm for calculating integrals of hypergeometric type
  functions and its realization in {REDUCE} system.
\newblock In {\em Proceedings of the international symposium on Symbolic and
  algebraic computation - {ISSAC} {\textquotesingle}90}. {ACM} Press.

\bibitem[\protect\citeauthoryear{Azzalini and Capitanio}{Azzalini and
  Capitanio}{1999}]{azzalini1999statistical}
Azzalini, A. and A.~Capitanio (1999).
\newblock Statistical applications of the multivariate skew normal
  distribution.
\newblock {\em Journal of the Royal Statistical Society: Series B (Statistical
  Methodology)\/}~{\em 61\/}(3), 579--602.

\bibitem[\protect\citeauthoryear{Balakrishnan, Johnson, and Kotz}{Balakrishnan
  et~al.}{2016}]{kotz2016continuous}
Balakrishnan, N., N.~Johnson, and S.~Kotz (2016).
\newblock {\em Continuous Univariate Distributions}.
\newblock Wiley Series in Probability and Statistics. John Wiley \& Sons
  Incorporated.

\bibitem[\protect\citeauthoryear{{Bithas}, {Sagias}, {Mathiopoulos},
  {Karagiannidis}, and {Rontogiannis}}{{Bithas} et~al.}{2006}]{1633320}
{Bithas}, P.~S., N.~C. {Sagias}, P.~T. {Mathiopoulos}, G.~K. {Karagiannidis},
  and A.~A. {Rontogiannis} (2006).
\newblock On the performance analysis of digital communications over
  generalized-k fading channels.
\newblock {\em IEEE Communications Letters\/}~{\em 10\/}(5), 353--355.

\bibitem[\protect\citeauthoryear{Boulogeorgos, Trevlakis, Tegos, Papanikolaou,
  and Karagiannidis}{Boulogeorgos et~al.}{2020}]{boulogeorgos2020MLinNanoBio}
Boulogeorgos, A.-A.~A., S.~E. Trevlakis, S.~A. Tegos, V.~K. Papanikolaou, and
  G.~K. Karagiannidis (2020).
\newblock Machine learning in nano-scale biomedical engineering.
\newblock {\em IEEE Transactions on Molecular, Biological and Multi-Scale
  Communications\/}.

\bibitem[\protect\citeauthoryear{Compiani and Kitamura}{Compiani and
  Kitamura}{2016}]{10.1111/ectj.12068}
Compiani, G. and Y.~Kitamura (2016, 10).
\newblock {Using Mixtures in Econometric Models: A Brief Review and Some New
  Results}.
\newblock {\em The Econometrics Journal\/}~{\em 19\/}(3), C95--C127.

\bibitem[\protect\citeauthoryear{Gradshteyn and Ryzhik}{Gradshteyn and
  Ryzhik}{2014}]{Gradshteyn2014}
Gradshteyn, I.~S. and I.~M. Ryzhik (2014).
\newblock {\em Table of Integrals, Series, and Products}.
\newblock Elsevier LTD, Oxford.

\bibitem[\protect\citeauthoryear{Gupta and Brown}{Gupta and
  Brown}{2001}]{gupta2001reliability}
Gupta, R.~C. and N.~Brown (2001).
\newblock Reliability studies of the skew-normal distribution and its
  application to a strength-stress model.
\newblock {\em Communications in Statistics-Theory and Methods\/}~{\em
  30\/}(11), 2427--2445.

\bibitem[\protect\citeauthoryear{Jakeman and Pusey}{Jakeman and
  Pusey}{1976}]{jakeman1976model}
Jakeman, E. and P.~Pusey (1976).
\newblock A model for non-rayleigh sea echo.
\newblock {\em IEEE Transactions on antennas and propagation\/}~{\em 24\/}(6),
  806--814.

\bibitem[\protect\citeauthoryear{Johnson, Kotz, and Balakrishnan}{Johnson
  et~al.}{1994}]{BookKotz}
Johnson, N.~L., S.~Kotz, and N.~Balakrishnan (1994).
\newblock {\em Continuous Univariate Distributions}.
\newblock Wiley-Interscience.

\bibitem[\protect\citeauthoryear{{Karagiannidis}, {Papathanasiou},
  {Diamantoulakis}, {Saratzis}, and {Saratzis}}{{Karagiannidis}
  et~al.}{2019}]{8941853}
{Karagiannidis}, G.~K., A.~{Papathanasiou}, P.~D. {Diamantoulakis},
  A.~{Saratzis}, and N.~{Saratzis} (2019).
\newblock A low complexity and cost method to diagnose arterial stenosis using
  lightwave wearables.
\newblock In {\em 2019 IEEE 19th International Conference on Bioinformatics and
  Bioengineering (BIBE)}, pp.\  675--680.

\bibitem[\protect\citeauthoryear{Liu, Liu, Liu, and Consalvi}{Liu
  et~al.}{2020}]{liu2020effects}
Liu, Y., G.~Liu, F.~Liu, and J.-l. Consalvi (2020).
\newblock Effects of the k-value solution schemes on radiation heat transfer
  modelling in oxy-fuel flames using the full-spectrum correlated
  k-distribution method.
\newblock {\em Applied Thermal Engineering\/}~{\em 170}, 114986.

\bibitem[\protect\citeauthoryear{Mudholkar and Hutson}{Mudholkar and
  Hutson}{2000}]{MUDHOLKAR2000291}
Mudholkar, G.~S. and A.~D. Hutson (2000).
\newblock The epsilon–skew–normal distribution for analyzing near-normal
  data.
\newblock {\em Journal of Statistical Planning and Inference\/}~{\em 83\/}(2),
  291--309.

\bibitem[\protect\citeauthoryear{Pedersen, Manch{\'o}n, Badiu, Shutin, and
  Fleury}{Pedersen et~al.}{2015}]{pedersen2015sparse}
Pedersen, N.~L., C.~N. Manch{\'o}n, M.-A. Badiu, D.~Shutin, and B.~H. Fleury
  (2015).
\newblock Sparse estimation using bayesian hierarchical prior modeling for real
  and complex linear models.
\newblock {\em Signal processing\/}~{\em 115}, 94--109.

\bibitem[\protect\citeauthoryear{Prudnikov}{Prudnikov}{1986}]{B:PrudnI}
Prudnikov, A. (1986).
\newblock {\em Integrals and Series: Volume 1: Elementary Functions; Volume 2:
  Special Functions}.
\newblock Taylor \& Francis.

\bibitem[\protect\citeauthoryear{Prudnikov, Brychkov, Brychkov, and
  Marichev}{Prudnikov et~al.}{1986}]{B:PrudnIII}
Prudnikov, A., Y.~Brychkov, I.~Brychkov, and O.~Marichev (1986).
\newblock {\em Integrals and Series: More special functions}.
\newblock Integrals and Series. Gordon and Breach Science Publishers.

\bibitem[\protect\citeauthoryear{{Trevlakis}, {Boulogeorgos},
  {Chatzidiamantis}, {Karagiannidis}, and {Lei}}{{Trevlakis}
  et~al.}{2020}]{TrevlakisOvsE}
{Trevlakis}, S.~E., A.~A.~A. {Boulogeorgos}, N.~D. {Chatzidiamantis}, G.~K.
  {Karagiannidis}, and X.~{Lei} (2020).
\newblock Electrical vs optical cell stimulation: A communication perspective.
\newblock {\em IEEE Access\/}~{\em 8}, 192259--192269.

\bibitem[\protect\citeauthoryear{Wang, Zhao, Wang, Zhang, and Alouini}{Wang
  et~al.}{2019}]{wang2019secrecy}
Wang, Z., H.~Zhao, S.~Wang, J.~Zhang, and M.-S. Alouini (2019).
\newblock Secrecy analysis in swipt systems over generalized-$ k $ fading
  channels.
\newblock {\em IEEE Communications Letters\/}~{\em 23\/}(5), 834--837.

\bibitem[\protect\citeauthoryear{Withers and Nadarajah}{Withers and
  Nadarajah}{2012}]{withers2012generalized}
Withers, C.~S. and S.~Nadarajah (2012).
\newblock A generalized suzuki distribution.
\newblock {\em Wireless Personal Communications\/}~{\em 62\/}(4), 807--830.

\bibitem[\protect\citeauthoryear{{Zhao}, {Tan}, {Pan}, {Chen}, and
  {Yang}}{{Zhao} et~al.}{2016}]{7406765}
{Zhao}, H., Y.~{Tan}, G.~{Pan}, Y.~{Chen}, and N.~{Yang} (2016).
\newblock Secrecy outage on transmit antenna selection/maximal ratio combining
  in mimo cognitive radio networks.
\newblock {\em IEEE Transactions on Vehicular Technology\/}~{\em 65\/}(12),
  10236--10242.

\bibitem[\protect\citeauthoryear{Zhou, Fang, Cristea, Lin, Tsai, Wan, Yeow, Ho,
  and Tsui}{Zhou et~al.}{2020}]{ZHOU2020106001}
Zhou, Z., J.~Fang, A.~Cristea, Y.-H. Lin, Y.-W. Tsai, Y.-L. Wan, K.-M. Yeow,
  M.-C. Ho, and P.-H. Tsui (2020).
\newblock Value of homodyned k distribution in ultrasound parametric imaging of
  hepatic steatosis: An animal study.
\newblock {\em Ultrasonics\/}~{\em 101}, 106001.

\end{thebibliography}

\end{document}